\documentclass[leqno,12pt]{amsart} 
\usepackage{amssymb}
\theoremstyle{plain} 
\newtheorem{theorem}{\indent\sc Theorem}[section]
\newtheorem{lemma}[theorem]{\indent\sc Lemma}
\newtheorem{corollary}[theorem]{\indent\sc Corollary}
\newtheorem{proposition}[theorem]{\indent\sc Proposition}

\theoremstyle{definition} 
\newtheorem{definition}[theorem]{\indent\sc Definition}
\newtheorem{remark}[theorem]{\indent\sc Remark}
\newtheorem{example}[theorem]{\indent\sc Example}

\usepackage[active]{srcltx}
\usepackage{pdfsync}
\usepackage{tabularx} 
\usepackage{arydshln}

\newcommand{\CC}{\mathbb{C}}
\newcommand{\C}{\mathcal{C}}
\newcommand{\FF}{\mathbb{F}}
\newcommand{\ZZ}{\mathbb{Z}}
\newcommand{\HH}{\mathbb{H}}
\newcommand{\OO}{\mathbb{O}}
\newcommand{\RR}{\mathbb{R}}

\newcommand{\Ric}{\mathrm{Ric}}
\newcommand{\Lie}{\mathrm{Lie}}

\newcommand{\SU}{\mathrm{SU}}

\def\dim{\mathop{\hbox{\rm dim}}}
\def\tr{\mathop{\rm tr}}
  
\newcommand{\spf}{\mathfrak{sp}}

\newcommand{\inder}{\mathfrak{inder}}

\newcommand{\sof}{\mathfrak{so}}
\newcommand{\slf}{\mathfrak{sl}}
\newcommand{\gl}{\mathfrak{gl}}

\newcommand{\id}{\mathrm{id}}
\newcommand{\mm}{\mathfrak{m}}
\newcommand{\hh}{\mathfrak{h}}
\newcommand{\g}{\mathfrak{g}}

\title[  A new family of homogeneous Einstein manifolds]{  A new family of homogeneous Einstein manifolds \\based on symplectic triple systems}

\author[C.~Draper]{ 
Cristina Draper Fontanals${}^*$ 
}  

\thanks{${}^*$ Supported by the Spanish Ministerio de Econom\'{\i}a y Competitividad---Fondo Europeo de
Desarrollo Regional (FEDER) MTM2016-76327-C3-1-P, and by the Junta de Andaluc\'{\i}a grant FQM-336, with FEDER funds} 

\address{%
Departamento de Matem\'{a}tica Aplicada,  Escuela de Ingenier\'\i as Industriales,\endgraf
Universidad de M\'{a}laga, 
 29071 M\'{a}laga,  
Spain
}
\email{cdf@uma.es}
\begin{document}


\maketitle 


\begin{abstract}
For each  simple symplectic triple system over the real numbers, the standard enveloping Lie algebra and the algebra of inner derivations of the triple provide a  reductive pair related to a semi-Riemannian homogeneous manifold. 
It is proved that   this   is an Einstein manifold. 
 \end{abstract}

\section{Introduction }

This note arises from the Second International Workshop on Nonassociative algebras held in Porto, May 2019. My talk there dealt with some results on 3-Sasakian manifolds \cite{nues3Sas} and how their close relationship with complex symplectic triple systems could bolster the study of some concrete questions related to   curvature and holonomy \cite{hol}. The fact that, for a real symplectic triple system, the standard enveloping algebra is never a compact Lie algebra, made me wonder how the geometry of the related manifold could be.  This geometry is of course quite different from 3-Sasakian geometry, since 3-Sasakian manifolds are Riemannian and compact. Nevertheless, the  algebraic relation   should have   geometric implications. 
To be precise, the
 first intuition is that these manifolds may be Einstein, just as 3-Sasakian manifolds are.
Indeed, this is the case, which can be proved using only algebraic properties of the non-associative structure involved, i.e.  that of the symplectic triple system. This is the aim of this note, namely to provide a purely algebraic proof that these manifolds are Einstein, that is, their Ricci tensor is multiple of the metric.

An  exhaustive  survey  of  results  concerning  Einstein metrics on compact manifolds of dimension at least four is found in the classical monograph \cite{Besse}.  Besse group tries to answer the following question: \emph{are there any best (or nicest, or distinguished) Riemannian structures on [a manifold] $M$?} 
 Note that, for surfaces, the best metrics are those  of constant (Gaussian) curvature. For dimension greater than 2, a good generalization is being of constant Ricci curvature. Recall that the Ricci curvature of a semi-Riemannian manifold is a quadratic differential form defined by the trace of the curvature tensor. The choice of an Einstein metric as privileged in some way can be justified both mathematically and physically. On one hand, there is a limited number of such metrics, on the other hand, Einstein metrics are the solutions to the Euler-Lagrange equations for stationary points of the integral of the scalar curvature. A key result is that every $G$-homogeneous Riemannian manifold of irreducible isotropy is Einstein \cite{Wolf68}. Also interesting in this regard is the complete classification of simply connected $G$-homogeneous Riemannian manifold ($G$ a connected compact simple Lie group) such that the metric induced by the Killing form is Einstein \cite{Ziller85}. There is both a regular interest and a large number of recent works on homogeneous Einstein metrics, see \cite{Alek} for a review. To cite just a few results, \cite{Kerr} has proved that every compact simply connected homogeneous space of dimension at most 11 admits a homogeneous Einstein metric, whereas   the compact manifold $\SU(4)/\SU(2)$ is a counterexample in  dimension 12  as shown in \cite{dim12}. In this latter work, the techniques to prove the existence   of homogeneous Einstein metrics follow a variational  approach, by searching  critical points of the total scalar curvature functional.
A variety of works about homogeneous Einstein metrics has been contributed by Arvanitoyeorgos and his collaborators. For instance, \cite{Arvani3} finds which 
  compact simple Lie groups  
  admit non-naturally reductive Einstein metrics, 
  \cite{Arvaniflag} constructs explicit invariant non-K\"ahler Einstein metrics on generalized flag manifolds, and
    \cite{Arvani2}
   obtains new invariant Einstein metrics on the quaternionic  Stiefel  manifold   of all orthonormal $p$-frames in $\HH^n$.
   
   In contrast, not much has been settled on homogeneous non-compact Einstein manifolds. One of the most cited papers on  this topic  is \cite{Heber},
  which presents a systematic study of the existence and moduli of homogeneous Einstein metrics with negative scalar curvature. Other works to be pointed out are \cite{Tamaru2}  on non-compact homogeneous Einstein manifolds attached to graded Lie algebras   or  \cite{Tamaru}  on the solvmanifolds which are  naturally homogeneous submanifolds of symmetric spaces of non-compact type. Anyway, all these examples are Riemannian.

 The family of manifolds that are studied in this work are semi-Riemannian, thus constituting a broad field to explore.   However, the choice of such family within this broad field, far from being arbitrary, stems from its close relationship with  (Riemannian) 3-Sasakian manifolds, which are, in turn, very well known \cite{libroGB}. It is worth mentioning that our tools for finding this family of Einstein metrics differ completely from the ones in the above cited works. Such previous techniques include combinatorial arguments, computational methods such as Gr\"obner basis, and  some indirect arguments too. In contrast, we resort to a direct proof that only uses algebraic properties of the related ternary structure. An appealing feature of the work is that all the cases are dealt with in a unified approach.  
 Besides, as the trace is invariant by complexification, some (well-known) results on Einstein metrics on   compact 3-Sasakian manifolds \cite{kashiwada} could be recovered from our work.

 We would like to emphasize the role of non-associative algebras (and non-associative structures, in general) in the study of Geometry. Of course,  the interplay between Algebra and Geometry in the theory of homogeneous spaces  is obvious,  but it is not so clear the great amount of non-associative structures hidden after some geometric objects, in particular of ternary type.  Without being exhaustive, some examples to be remarked could be Jordan and Lie triple systems  \cite{libroJordan},  or Lie-Yamaguti algebras \cite{Fabi}. Also, some relations between ternary non-associative structures in general (and symplectic triple systems in particular) and physics are explained in \cite{kerner}. In our work, Freudenthal triple systems \cite{Meyberg} will appear   in connection to a new tensor related to the curvature (see Remark~\ref{re_freu}).\smallskip

   This work is structured as follows. Section~\ref{se_alg} describes the symplectic triple system structure, as well as two important attached Lie algebras, namely, the standard enveloping algebra and the Lie algebra of the inner derivations of the triple. These two pieces will be, respectively, the Lie algebras of the groups $G$ and $H$ of our homogeneous manifolds $G/H$. It is also listed a collection of examples of simple symplectic triple systems  that this theory can be applied to, precisely those with split standard enveloping algebra. Section~\ref{se_geo} deals with the related manifolds, starting with the introduction of a $G$-invariant metric inspired in the 3-Sasakian case. Thereupon the computations in \cite{hol} can be used to have a complete description of the Riemannian curvature tensor in terms of the triple product and of the symplectic form attached to the symplectic triple system. This algebraic expression of the curvature tensor is very convenient for our purposes.
 The key is to introduce a new tensor $Q$ in Lemma~\ref{le_defQ} which measures how far  the manifolds are from being of constant curvature. This tensor turns out to have zero trace in the manifolds under study (Proposition~\ref{pr_lacuenta} (b)), which equivales to the fact that the Ricci tensor is multiple of the metric (Corollary~\ref{main}). All along this note our approach  is   algebraic.

\section{The algebraic structure }\label{se_alg}

\subsection{Symplectic triple systems }\label{se_sts}

The ternary structure of symplectic triple system made its first appearance in \cite{ternarias}, and its standard enveloping algebra in \cite{alb_sts}. The material here is mainly extracted from \cite{hol}. 

\begin{definition}\label{def1}
Let   $T$ be a real vector space endowed with 
a non-zero alternating form $(\, ,\,  )\colon T\times T\to\mathbb R$, 
  and a triple product $[\, ,\, ,\,]\colon T\times T\times T\to T$.
  It is said that $(T,[\, ,\, ,\,],(\, ,\,  ))$ is a \emph{symplectic triple system} if satisfies 
\begin{align}
[x,y,z]=[y,x,z],\label{eq_uno}\\
[x,y,z]-[x,z,y]=(x,z)y-(x,y)z+2(y,z)x,\label{eq_dos}\\
[x,y,[u,v,w]]=[[x,y,u],v,w]+[u,[x,y,v],w]+[u,v,[x,y,w]],\label{eq_tres}\\
([x,y,u],v)=-(u,[x,y,v]),\label{eq_cuatro}
  \end{align}
  for any $x,y,z,u,v,w\in T$. 
  \end{definition} 
  An \emph{ideal} of the symplectic triple system $T$ is a subspace $I$ of $T$ such that $[T,T,I]+[T,I,T]\subset I$ and the system is said to be  \emph{simple} if $[T,T,T]\ne0$ and it contains no proper ideal. 
  The simplicity of $T$ is equivalent to the non-degeneracy of the bilinear form \cite[Proposition~2.4]{alb_triples}.\footnote{In this reference,    $\dim T>2$ is required, but this is only used in the proof in the case the ground field has characteristic 3. Such hypothesis is not necessary if the field is $\RR$. } 
   If   $d\in \gl(T)=\mathop{\rm{End}}_{\FF}(T)^-$, $d$ is called a \emph{derivation} if
$$
d([u,v,w])=[d(u),v,w]+[u,d(v),w]+[u,v,d(w)],
$$
for all $u,v,w\in T$. For instance, Eq.~\eqref{eq_tres} tells that  the map $d_{x,y}\colon T\to T$, $d_{x,y}(z):=[x,y,z]$, is a derivation for any $x,y\in T$. Moreover, Eq.~\eqref{eq_tres} also implies  that 
$$
  \inder(T):=\{\sum_{i=1}^nd_{x_i,y_i}:x_i,y_i\in T,n\in\mathbb{N}\},
  $$
  is a Lie subalgebra of the general linear algebra $\gl(T)$, since 
  $
  [d_{x,y},d_{u,v}]=d_{[x,y,u],v}+d_{u,[x,y,v]}\in \inder(T).
  $
This receives the name of algebra of \emph{inner derivations}.   
Now, consider $V=\RR e_1\oplus\RR e_2$ a two-dimensional vector space with the  non-zero alternating bilinear form $\langle.,.\rangle\colon V\times V\to\RR$ given by $\langle e_1,e_2\rangle=1$. Note that, by identifying  each 
$\gamma \in \spf(V,\langle.,.\rangle)=\{\gamma\in\gl(V): \langle \gamma(a),b\rangle +\langle a,\gamma(b)\rangle =0\ \forall a,b\in V\}$ with its matrix relative to the basis $\{e_1,e_2\}$ of $V$, the symplectic Lie algebra 
$\spf(V,\langle.,.\rangle)$ coincides with $\mathfrak{sl}_2(\RR)$, because
$$
\gamma_{e_1,e_2}=\begin{pmatrix}-1&0\\0&1\end{pmatrix}; \quad
\gamma_{e_1,e_1}=\begin{pmatrix}0&2\\0&0\end{pmatrix};\quad
\gamma_{e_2,e_2}=\begin{pmatrix}0&0\\-2&0\end{pmatrix};
$$
being $\gamma_{a,b}:=\langle a,.\rangle b+\langle b,.\rangle a\in \spf(V,\langle.,.\rangle)$. Consider the vector space
$$
  \g(T):=\spf(V,\langle.,.\rangle)\oplus \inder(T)\oplus\, V\otimes T,
  $$
  and endow it with the following skew-symmetric product, when extending by linearity:
   \begin{itemize}
  \item  $[\gamma +d,\gamma' +d']:=[\gamma,\gamma']+[d,d']$;
  \item $[\gamma +d, a\otimes x]:=\gamma(a)\otimes x+a\otimes d(x)$;
  \item $[a\otimes x,b\otimes y]:=(x,y)\gamma_{a,b}+\langle a,b\rangle d_{x,y}$;
  \end{itemize}
  where $\gamma,\gamma'\in\spf(V,\langle.,.\rangle)$, $d,d'\in  \inder(T)$, $a,b\in V$ and $x,y\in T$.
  Thus $( \g(T),[.,.])$ is a  $\ZZ_2$-graded Lie algebra, with homogeneous components given by
\begin{equation}\label{eq_grad}
  \g(T)_{\bar0}:=\spf(V,\langle.,.\rangle)\oplus \inder(T);\qquad \g(T)_{\bar1}:=  V\otimes T.
\end{equation}
The algebra $ \g(T)$ is called the \emph{standard enveloping algebra} related to the symplectic triple $T$. 
  Moreover, $ \g(T)$ is a simple Lie algebra if and only if so is $(T,[\, ,\, ,\,],(\, ,\,))$ (\cite[Theorem~2.9]{alb_triples}).
The reductive pairs under study will be
$( \g(T),\inder(T))$, for $T$ a simple 
 symplectic triple system over the reals, where the chosen complementary subspace is the orthogonal to $\mathfrak{h}= \inder(T)$ with respect to the Killing form, that is,
$$
\mathfrak{m}=\spf(V,\langle.,.\rangle)\oplus V\otimes T.
$$ 
 
All the above definitions and results work analogously for the complex field. Note that, 
if $T$ is a real     symplectic triple system, then $T^\CC=T\otimes_\RR\CC$ is a complex   symplectic triple system, since it satisfies  the properties $(1)$-...-$(4)$ in Def.~\ref{def1}. Moreover, if $T$ is simple, then $T^\CC$ is simple, since the bilinear form $(\, ,\,  )\colon T^\CC\times T^\CC\to\mathbb C$ is non-degenerate just when $(\, ,\,  )$ so is. This implies that the standard enveloping algebra  $ \g(T)$ is not only simple but central simple, i.e., $ \g(T)^\CC$ is a simple complex Lie algebra. This will allow us to use the computations in \cite{nues3Sas,hol}, since, for $3$-Sasakian homogeneous manifolds $M=G/H$, there exists a complex simple symplectic triple system $W$ such that $(\Lie(G)^\CC,\Lie(H)^\CC)\cong (\g(W),\inder(W))$.

 
 \subsection{Some key examples}
 
 The classification of the symplectic triple systems is well-known for the complex field,  
 but this is not the case for the real one. A very recent classification is achieved in \cite{conAlb}, where precisely a key tool is the fact that $T\otimes_\RR\CC$ is a complex simple symplectic triple system when 
   $T$ is a real   simple symplectic triple system.
   
    Here we focus on 
   some important examples of    simple  symplectic triple systems over $\RR$, directly adapted from the    examples over $\CC$, precisely those whose  standard enveloping algebra is split. There appears just one such  real   simple symplectic triple system $T$   related to each  complex simple symplectic triple system.    That family of examples with split standard enveloping algebra is exhibited here in order to have available some examples of homogeneous manifolds this theory can be applied to, although this is not a complete list according to  \cite{conAlb}. 
   
   We call  a   symplectic triple system $T$ of \emph{symplectic} (respectively of  orthogonal, special, exceptional) \emph{type} when the standard enveloping algebra is a symplectic   (respectively  orthogonal, special, exceptional) Lie algebra.

 \begin{example}\label{ex_symplectic}
Consider the non-degenerate alternating bilinear form $(\,,\,)\colon \RR^{2n}\times \RR^{2n}\to\RR$ given by
\begin{equation}\label{eq_forma}
\left(\begin{pmatrix}x_1\\x_2\end{pmatrix},\begin{pmatrix}y_1\\y_2\end{pmatrix}\right)=x_1\cdot y_2-x_2\cdot y_1,
\end{equation}
if $x_1,x_2,y_1,y_2\in\RR^n$, for $\cdot$ the usual scalar product in $\RR^n$.
Then $T=\RR^{2n}$ is a symplectic triple system with the triple product given by
$$
[x,y,z]:=(x,z)y+(y,z)x,
$$
if $x,y,z\in T$. It is easy to check that the algebra of inner derivations $\inder(T)$ is isomorphic to 
$$ \begin{array}{ll} 
\spf_{2n}(\RR)&=\{f\in\mathfrak{gl}_{2n}(\RR):(f(x),y)+(x,f(y))=0\,\forall x,y\in\RR^{2n}\}\vspace{3pt}\\
&=\left\{\tiny\begin{pmatrix} A&B\\C&-A^t\end{pmatrix}:B=B^t, C=C^t, A,B,C\in\mathfrak{gl}_{n}(\RR)\right\};
\end{array}
$$
and the standard enveloping algebra $\g(T)$ is isomorphic to the symplectic Lie algebra $\spf_{2n+2}(  \RR)$, 
where   the alternating bilinear form in $V\oplus \RR^{2n}\cong\RR^{2n+2}$ is defined as
$$
(u+x,v+y):=\langle u,v\rangle +(x,y),
$$
if $u,v\in V\cong\RR^{2}$, $x,y\in \RR^{2n}$.
This symplectic triple system $T=\RR^{2n}$ with  $(\g(T),\inder(T))\cong(\spf_{2n+2}(\RR),\spf_{2n}(\RR))$   has  \emph{symplectic type}. 
\end{example}

\begin{example}\label{ex_orthogonal}
Let   $b\colon \RR^{2n}\times \RR^{2n}\to\RR$ be the 
non-degenerate symmetric bilinear form $b(x,y)=x^t \tiny\begin{pmatrix}0&I_n\\I_n&0\end{pmatrix}y$.
 Then $T=\RR^{4n}$ is a symplectic triple system with alternating form
$$
\left(\tiny\begin{pmatrix}x\\x'\end{pmatrix},\tiny\begin{pmatrix}y\\y'\end{pmatrix}\right):=
\frac12 \big( b(x,y')-b(x',y)\big),
$$
and triple product given by
$$
\left[\tiny\begin{pmatrix}x\\x'\end{pmatrix},
\tiny\begin{pmatrix}y\\y'\end{pmatrix},
\tiny\begin{pmatrix}z\\z'\end{pmatrix}\right]:= 
\tiny\begin{pmatrix} -\frac12\big(b(x,y')+b(x',y)\big)z+b(x,y)z'+b(x,z)y'-b(y',z)x-b(x',z)y+b(y,z)x'\\\frac12\big(b(x,y')+b(x',y)\big)z'-b(x',y')z+b(x,z')y'-b(y',z')x-b(x',z')y+b(y,z')x' \end{pmatrix}
$$
for any   $x,x',y,y',z,z'\in \RR^{2n}$. Besides,    the algebra of inner derivations $\inder(T)$ is isomorphic to $  \slf_2(\RR)\oplus\sof_{n,n}(\RR)$
and the standard enveloping algebra $\g(T)$ is isomorphic to the orthogonal algebra $\sof_{n+2,n+2}(\RR)$.
This   symplectic triple system  $T=\RR^{4n}$ has \emph{orthogonal type}, with  $(\g(T),\inder(T))\cong(\sof_{n+2,n+2}(\RR),\slf_2(\RR)\oplus\sof_{n,n}(\RR))$. 
\end{example}

 \begin{example}\label{ex_special}
Take   $T=\RR^{2n}$ with the alternating form as in Eq.~\eqref{eq_forma}, but with the triple product defined by
$$
\left[\begin{pmatrix}x_1\\x_2\end{pmatrix},
\begin{pmatrix}y_1\\y_2\end{pmatrix},
\begin{pmatrix}z_1\\z_2\end{pmatrix}
\right]= 
\begin{pmatrix}-2(z_1\cdot y_2)x_1-2(z_1\cdot x_2)y_1-(x_1\cdot y_2+x_2\cdot y_1)z_1\\
2(z_2\cdot y_1)x_2+2(z_2\cdot x_1)y_2+(x_1\cdot y_2+x_2\cdot y_1)z_2\end{pmatrix}.
$$
 This is again a simple symplectic triple system with 
 $(\g(T),\inder(T))\cong(\slf_{n+2}(\RR), \gl_n(\RR))$,
of  \emph{special type}. 
\end{example}

\begin{example}\label{ex_exceptional2}
Let $C$ be a real (finite-dimensional) unital composition algebra, that is, endowed with a non-degenerate quadratic form $N\colon C\to\RR$ such that $N(xy)=N(x)N(y)$. Then, either $C$ is a real
division algebra (belonging to $ \{\RR,\CC,\HH, \OO\}$), or $C$ is isomorphic, with $N(x)=x\bar x$, to one of the next examples:
\begin{itemize}
\item $\RR^2$, for $\overline{(x,y)}=(y,x)$;
\item $\textrm{Mat}_{2\times2}(\RR)$ with $\overline{ \scriptsize\begin{pmatrix}\alpha&x\\y&\beta  \end{pmatrix}} =\scriptsize\begin{pmatrix}\beta&-x\\-y&\alpha\end{pmatrix}$;
\item The \emph{Zorn matrix algebra}, that is, $ \left\{ \scriptsize\begin{pmatrix}\alpha&x\\y&\beta\end{pmatrix}:\alpha,\beta\in\RR,\ x,y\in\RR^3  \right\}$ with product
$$
\scriptsize\begin{pmatrix}\alpha&x\\y&\beta\end{pmatrix}
\scriptsize\begin{pmatrix}\alpha'&x'\\y'&\beta'\end{pmatrix}=
\scriptsize\begin{pmatrix}\alpha\alpha'+x\cdot y'&\alpha x'+\beta'x+x\times x'\\\alpha' y+\beta y'-y\times y'&\beta
\beta'+y\cdot x'\end{pmatrix}
$$
and involution $\overline{ \scriptsize\begin{pmatrix}\alpha&x\\y&\beta  \end{pmatrix}} =\scriptsize\begin{pmatrix}\beta&-x\\-y&\alpha\end{pmatrix}$, where $\cdot$ and $\times$ denote the usual dot and cross products in $\RR^3$.  It is also called split-octonion algebra and denoted by $\mathbb O_s$.
\end{itemize}
These three algebras have non-trivial idempotents, so that they are not division  algebras.
Assume $C$ is either $\RR$ or one of these non-division composition algebras. Take $ H_3(C )=\{x=(x_{ij})\in \mathop{\textrm{Mat}}_{3\times3}(C ):x_{ji}=\overline{x_{ij}}\}$ the set of hermitian matrices with coefficients on $C$, which turns out to be a Jordan algebra with the symmetrized product $x\circ y=\frac12(xy+yx)$.
  Then the vector space 
$$
T_C=\left\{\begin{pmatrix}\alpha& a\\b&\beta\end{pmatrix}:\alpha,\beta\in\RR,a,b\in H_3(C )\right\}
$$
becomes a simple symplectic triple system of \emph{exceptional type} with the alternating map and triple product given as in \cite[Example~2.6]{hol}.
In this way   the pair $(\g(T),\inder(T))$ turns out to be isomorphic, respectively, to 
$$
(\mathfrak{f}_{4,4},  \spf_6(\RR) ),\quad
(\mathfrak{e}_{6,6}, \slf_6(\RR) ),\quad
(\mathfrak{e}_{7,7}, \mathfrak{so}_{6,6}(\RR) ),\quad
(\mathfrak{e}_{8,8},\mathfrak{e}_{7,7}).
$$

The following symplectic triple system has exceptional type too.
Take $V_n=\RR_n[X,Y]$ the linear space of the degree $n$ homogeneous polynomials in two variables $X,Y$.
For any $f\in V_n$, $g\in V_m$, the transvection $(f,g)_q\in V_{m+n-2q}$ is defined by $(f,g)_q=0$ if $q>\textrm{min}(n,m)$, and
$$
(f,g)_q=\frac{(n-q)!}{n!}\frac{(m-q)!}{m!}\sum_{i=0}^q(-1)^i\binom{q}{i}\frac{\partial^qf}{\partial x^{q-i}\partial y^{i}}
\frac{\partial^qg}{\partial x^{i}\partial y^{q-i}}
$$
otherwise. Then $T=V_3$, with the alternating form given by $(f,g):=(f,g)_3$ and the triple product given by
$$
[f,g,h]:=6((f,g)_2,h)_1,
$$
is a symplectic triple system with related reductive pair $(\g(T),\inder(T))=(\mathfrak{g}_{2,2},\slf_2(\RR))$.
In this case, the corresponding homogeneous manifold has been considered in \cite[\S5.4]{esphomogeneosdeG2}, but    the metric considered there, induced by the Killing form, is not an Einstein metric.
\end{example}

\section{The geometric structure }\label{se_geo}

 If $G$ is any  Lie group with related Lie algebra $ \g(T)$, it is well-known that there is a unique connected subgroup  $H$ of $G$ whose  related Lie algebra is $\mathfrak h=\inder(T)$. In case $H$ is closed, or equivalently $H$ has the  topology induced from $G$, 
 then $M=G/H$ is a   homogeneous space. This will always be the case. 
 It is difficult to have general arguments for checking that a subgroup is closed only with the {\it infinitesimal}  information about its Lie algebra. But, in our case, $G$ and $H$ can be taken not only as  matrix Lie groups, but also as  algebraic groups.

 Accordingly with our previous list of (split) examples, we get a (not complete) list of homogeneous manifolds this theory can be applied to. For instance, focusing on the classical examples,   
 \begin{equation}\label{eq_manifolds}
 \rm{Sp}_{2n+2}(\RR)/\rm{Sp}_{2n}(\RR);\qquad 
 {\rm{SO}}^+_{n+2,n+2}(\RR)/{\rm{SO}}^+_{n,n}(\RR)\times  {\rm{SL}}_{2}(\RR);\qquad 
 {\rm{SL}}_{m+2}(\RR)/ {\rm{GL}}_m^+(\RR).
 \end{equation}
 \begin{itemize}
 \item Symplectic case: Take the vector  space $\mathcal{U}=\mathbb{R}^{2n}$ endowed with a non-degenerate alternating form $\omega\colon \mathcal{U}\times\mathcal{U}\to\RR$, and $V=\RR^2$ with the symplectic form $\langle.,.\rangle$ as in Section~\ref{se_sts}. Consider  the group $G= \rm{Sp}(\mathcal{U}\oplus V,\omega\perp\langle.,.\rangle )\cong \rm{Sp}_{2n+2}(\RR)$. 
The (obviously closed) subgroup $H=\{\sigma \in G:\sigma \vert_V=\id\}$ can be naturally  identified with $\rm{Sp}(\mathcal{U},\omega )\cong {\rm{Sp}}_{2n}(\RR)$, by means of the map $\sigma \mapsto \sigma \vert_\mathcal{U} $.\smallskip

 \item Orthogonal case: Take the vector space $\mathcal{U}=\mathbb{R}^{2n}$ endowed with a non-degenerate bilinear symmetric form $b\colon \mathcal{U}\times\mathcal{U}\to\RR$ of neutral signature. Take $\HH_s={\rm{Mat}}_{2\times2}(\RR)$ as in Example~\ref{ex_exceptional2}, with the symmetric form $N$ -of neutral signature too- whose related quadratic form is given by the determinant. Consider
 the group $G= {\rm{SO}}(\mathcal{U}\oplus\HH_s,b\perp N)\cong {\rm{SO}}_{n+2,n+2}(\RR)$. 
 The closed  subgroup $ \{\sigma \in G:\sigma (\mathcal{U})\subset  \mathcal{U},\sigma (\HH_s)\subset  \HH_s\}$ can be naturally  identified with ${\rm{S}}({\rm{O}}_{n,n}(\RR)\times {\rm{O}}_{2,2}(\RR))$  by means of the map 
 $\sigma \mapsto (\sigma \vert_\mathcal{U},\sigma \vert_{\HH_s}) $.
 Taking in mind that $  {\rm{SL}}_{2}(\RR)\times {\rm{SL}}_{2}(\RR)\cong{\rm{SO}}_{2,2}(\RR)$,
 $(a,b)\mapsto L_aR_b\colon \HH_s\to\HH_s$  (left and right  multiplication operators),
  the   subgroup we are interested in is 
 $$
 H= \{\sigma \in G:\sigma (\mathcal{U})\subset  \mathcal{U},\det(\sigma \vert_\mathcal{U})=1, \sigma (\HH_s)\subset  \HH_s,\sigma \tau_a=\tau_a\sigma \ \forall a\in  {\rm{SL}}_{2}(\RR)\},
 $$
  where 
 $\tau_a\in G$ is defined by $\tau_a\vert_\mathcal{U}=\id$ and $\tau_a\vert_{\HH_s}=L_a$.
 The group $H$ is closed since it is defined as the zero set of a smooth function (in fact, a  polynomial function), noting that commuting with $\tau_a$ for any $a\in  {\rm{SL}}_{2}(\RR)$ is equivalent to commute with 
 $\{\tau_{a_i}\}_{i=1,2}$, for 
  $$ 
  a_1=\tiny\begin{pmatrix}1&1\\0&1\end{pmatrix},\qquad a_2=\tiny\begin{pmatrix}1&0\\1&1\end{pmatrix}.
  $$ 
It can be easily checked that $H\cong{\rm{SO}}_{n,n}(\RR)\times  {\rm{SL}}_{2}(\RR)$. 
Note that   $G= {\rm{SO}}_{n+2,n+2}(\RR)$ is not connected, so, in order  to agree with Eq.~\eqref{eq_manifolds}, take the identity component $G_0= {\rm{SO}}^+_{n+2,n+2}(\RR)$. In fact, 
  we can apply the theory both to $G/H$ and to $G_0/H_0$.
 \smallskip
 
 \item Special case: Take a vector subspace $\mathcal{U}=\mathbb{R}^{m}$, and, as before $V=\RR^2$.
 Consider the group $G=\{\sigma \in{\rm{GL}}(\mathcal{U}\oplus V):\det \sigma =1\}$, which is isomorphic to ${\rm{SL}}_{m+2}(\RR)$. The closed subgroup,   
 $$
 H=\{\sigma \in G: \sigma (\mathcal{U})\subset \mathcal{U},\, \det \sigma \vert_\mathcal{U}>0,\,
 \sigma \vert_{V}=\frac1{\sqrt{\det \sigma \vert_\mathcal{U}} }\id_V\}
 $$
  is isomorphic to 
 ${\rm{GL}}^+(\mathcal{U})\cong {\rm{GL}}_m^+(\RR)$, again  by means of the map $\sigma \mapsto \sigma \vert_\mathcal{U} $.
 \end{itemize}
\medskip

\subsection{Background }
Thus our situation is, from now on, a  Lie group $G$ with related Lie algebra $\g= \g(T)$, and a closed connected subgroup  $H$ of $G$ whose  related Lie algebra is $\mathfrak h= \inder(T)$, for some $T$     simple symplectic triple system over $\RR$. In particular, $M=G/H$ is a reductive homogeneous space.

 The differential map $\pi_{*}$ of the canonical projection $\pi\colon G\to M$ permits us to identify the tangent space $T_{\pi(e)}M$ with $\mathfrak{m}=\spf(V,\langle.,.\rangle)\oplus V\otimes T\le\g(T)\cong T_eG$,   by the linear isomorphism
$(\pi_{*})_e\vert_{\mathfrak{m}}\colon\mathfrak{m}\to  T_oM$, for $o=\pi(e)=eH\in M$. Thus we can work in an algebraic setting.
Let us denote by $[\ ,\ ]_{\mathfrak{h}}$ and   $[\ ,\ ]_{\mathfrak{m}}$   the composition of the bracket
$ [\,,\,]\colon \mathfrak{m}\times \mathfrak{m}\to \mathfrak{g}$ with  the projections 
$\pi_{\mathfrak{h}},\pi_{\mathfrak{m}}\colon \g\to\g$  of $\mathfrak{g}= \mathfrak {h}\oplus \mathfrak{m}$ onto each summand, respectively. Recall   the $\mathbb Z_2$-grading on $\g={\g_{\bar0} }\oplus {\g_{\bar1} }$   considered in Eq.~\eqref{eq_grad}.

Take the  $G$-invariant metric $g$ on $M$  determined by $ g_o\colon T_oM\times T_oM\to\RR$, which, under the identification $(\pi_{*})_e\vert_{\mathfrak{m}}$, is defined by 
\begin{equation}\label{eq_nuestrag}
g\vert_{\spf(V)}=- \frac1{4(n+2)} \kappa,\quad g\vert_{\g_{\bar1} }=-\frac1{8(n+2)}\kappa,\quad g\vert_{\spf(V)\times\g_{\bar1} }=0,
\end{equation}
for $\kappa$ the Killing form of $\g$. 
Observe that the choice of this metric is inspired in   the Einstein metric in any   3-Sasakian homogeneous manifold \cite[Theorem~4.3ii)]{nues3Sas}.
Our knowledge of the Killing form shows, as in \cite[Eq.~(20)]{hol},   the behavior of $g\vert_{\g_{\bar1} }$, that is, 
\begin{equation}\label{eq_godd}
g(a\otimes x,b\otimes y)=\frac12\langle a,b\rangle(x,y),
\end{equation}
for any $a,b\in V$, $x,y\in T$.
\smallskip

In order to work with  the Levi-Civita connection, the only metric affine connection with zero torsion, 
consider  the related map $\alpha^g\colon \mm\times\mm\to\mm$, which is the only $\hh$-invariant map  satisfying
$$
\alpha^{g} (X,Y)-\alpha^{g} (Y,X)=[X,Y]_\mathfrak{m}
$$
and 
$$
g(\alpha^{g} (X,Y),Z)+g(Y,\alpha^{g} (X,Z))=0
$$
for any $X,Y,Z\in\mm$. 
With similar arguments as in  \cite[Theorem~4.3]{nues3Sas}, such map is proved to be
\begin{equation}\label{eq_alfadeLevi}
 \alpha^g(X,Y)=\left\{\begin{array}{ll}
 0&\text{if $\,X\in\spf(V,\langle.,.\rangle)$ and $\,Y\in\g_{\bar1}$},\\
 \frac12[X,Y]_\mm&\text{if either $\,X,Y\in\spf(V,\langle.,.\rangle)$ or $\,X,Y\in\g_{\bar1}$},\\
{ [X,Y]}_\mm&\text{if $\,X\in\g_{\bar1}$ and $\,Y\in\spf(V,\langle.,.\rangle)$}.\\
 \end{array}
 \right.
\end{equation}

Now, the curvature tensor  can   be expressed in terms of this related map   as follows:
$$
 R (X,Y)Z=  
\alpha^{g} (X,\alpha^{g}(Y, Z))-\alpha^{g}(Y,\alpha^{g} (X, Z)) 
  -\alpha^{g}( [X , Y]_{\mathfrak{m}}, Z)-[[X , Y]_{\mathfrak{h}} , Z],
$$
for any $X, Y, Z\in \mathfrak{m}$.
More specifically, by \cite[Proposition~4.1]{hol},
\begin{equation}\label{eq_curv}
\begin{array}{l}
R(\xi,\xi') (\xi''+a\otimes x)=-\frac14[[\xi,\xi'],\xi''],\vspace{3pt}\\
R(a\otimes x, \xi)(\xi'+b\otimes y)=
-\frac12(x,y)\langle a,b\rangle \xi+g(\xi,\xi')a\otimes x,\vspace{4pt}\\
R(a\otimes x, b\otimes y)(\xi+c\otimes z)=\frac{\gamma_{a,c}(b)\otimes (x,z)y-\gamma_{b,c}(a)\otimes (y,z)x}2-\langle a,b \rangle c\otimes [x,y,z],
\end{array}
\end{equation}
for any $\xi,\xi',\xi''\in  \spf(V,\langle.,.\rangle)$, $a,b,c\in V$, $x,y,z\in T$, and recall that trilinearity allows us to know completely $R\colon\mm\times\mm\times\mm\to\mm$. 
This algebraic expression will be used next to prove that indeed $M$ is an Einstein manifold, since it is not difficult to compute the trace of the related map $R(-,X)Y$ in a unified way independently of the choice of the  symplectic triple system $T$.


\subsection{ Ricci tensor }

Recall that the Ricci tensor   is the symmetric $(0,2)$-tensor defined by $\Ric(X,Y)=\tr R(-,X)Y$. We compute this trace directly, by using the properties of the symplectic triple system. To that aim, we introduce a tensor $Q$ which measures the difference between the curvature tensor and certain metric operators. This will make easier the computations for the trace.

\begin{lemma}\label{le_defQ}
The curvature tensor equals
$$
R(X,Y)=g(Y,-)X-g(X,-)Y+Q(X,Y),
$$
where $Q(X,Y)\colon\mm\to\mm$, $Z\mapsto Q(X,Y,Z)$ ($X,Y,Z\in\mm$)
vanishes  if $X$, $Y$ or $Z$ belongs to $\spf(V)$, and is defined by linear extension on $V\otimes T$ by
\begin{equation*}\label{eq_nuevotensorQ}
Q(a\otimes x, b\otimes y,c\otimes z)=\langle a,b\rangle c\otimes ((x,z)y+(y,z)x-[x,y,z])
\end{equation*}
for any $a,b,c\in V$, $x,y,z\in T$.
\end{lemma}

The proof is straightforward following \cite[Eqs.~(18) and (19)]{hol} and the lines in \cite[Remark~4.3]{hol}. These computations are obtained as an immediate consequence of Eq.~\eqref{eq_curv}, where the algebraic expression of the curvature operators is described.

The key for checking that $g$ is always an Einstein metric is that the introduced tensor $Q$ will have zero trace. Observe that this tensor measures how far is the manifold from being of constant curvature, so that, for instance, $Q=0$ if (and only if) the symplectic triple system is of symplectic type (Example~\ref{ex_symplectic}).

\begin{proposition}\label{pr_lacuenta}
We have
\begin{itemize}
\item[a)] $\tr g(X,-)Y=g(X,Y)$ for any $X,Y\in\mm$; and
\item[b)] $\tr Q(X,-,X)=0$ for all $X\in V\otimes T$.
\end{itemize}
\end{proposition}

\begin{proof}  
Recall that if $(W,g)$  denotes a real vector space $W$ with a non-degenerate symmetric bilinear form $g\colon W\times W\to\RR$,   then the trace of any linear map  $f\colon W\to W$ can be computed from any orthogonal basis $\{E_i\}$ of $W$
as
$\tr(f)=\sum_ig(f(E_i),E_i)/g(E_i,E_i)$. 

Item a) is very well-known: we complete $0\ne\{Y\}$ until an orthogonal basis. Then the matrix of the map $f=g(X,-)Y$ relative to such a basis has only the first row different from zero, so its the trace is $g(X,Y)$ since $f(Y)=g(X,Y)Y$. 

For item b), take $\{x_i,y_i\}_{i=1,\dots,m}$ a basis of $T$ such that the symplectic form is given by
$$
(x_i,y_j)=\delta_{ij};\quad (x_i,x_j)=(y_i,y_j)=0,
$$ 
what is usually called a symplectic basis of $T$. (Here it is essential the simplicity of $T$.)
Recall   our choice of $e_1,e_2\in V$ with $\langle e_1,e_2\rangle=1$ used throughout this work. This provides the following orthogonal basis of $(V\otimes T,g\vert_{V\otimes T})$,
$$\begin{array}{ll}
E_{i}=e_1\otimes x_i+e_2\otimes y_i,\qquad&
E_{i+m}=e_1\otimes y_i-e_2\otimes x_i,\\
E_{i+2m}=e_1\otimes y_i+e_2\otimes x_i,\qquad&
E_{i+3m}=e_1\otimes x_i-e_2\otimes y_i;
\end{array}   
$$
for $i\in\{1,\dots,m\}$. It satisfies $g(E_i,E_i)=1$ if $i\in\{1,\dots,2m\}$ and $g(E_i,E_i)=-1$ if $i\in\{2m+1,\dots,4m\}$, taking in mind Eq.~\eqref{eq_godd}.
Take $X$ an arbitrary element in $V\otimes T$, which can be written as $X=e_1\otimes z+e_2\otimes t$ for some $z,t\in T$. Thus the trace of the map $Q(X,-,X)\colon\mm\to\mm$
is
$\tr Q(X,-,X)=S_0+S_1-S_2-S_3$, for $S_k=\sum_{i=1}^{m}g(Q(X,E_{i+km},X),E_{i+km})$, since 
$Q(X,\spf(V),X)=0$.
For convenience, denote by 
\begin{equation}\label{eq_llave}
\{x,y,z\}:=(x,z)y+(y,z)x-[x,y,z]
\end{equation}
 the trilinear map defined on $T$. Now, we compute, for $1\le i\le m$,
$$\begin{array}{l}
Q(X,E_i,X)=e_1\otimes\big(\{z,y_i,z\}-\{t,x_i,z\}\big)+e_2\otimes\big(\{z,y_i,t\}-\{t,x_i,t\}\big),\\
Q(X,E_{i+m},X)=-e_1\otimes\big(\{z,x_i,z\}+\{t,y_i,z\}\big)-e_2\otimes\big(\{z,x_i,t\}+\{t,y_i,t\}\big),\\
Q(X,E_{i+2m},X)=e_1\otimes\big(\{z,x_i,z\}-\{t,y_i,z\}\big)+e_2\otimes\big(\{z,x_i,t\}-\{t,x_i,t\}\big),\\
Q(X,E_{i+3m},X)=-e_1\otimes\big(\{z,y_i,z\}+\{t,x_i,z\}\big)-e_2\otimes\big(\{z,y_i,t\}+\{t,x_i,t\}\big).
\end{array}
$$
By following the computation, and taking into account Eq.~\eqref{eq_godd},  we have
$$
\begin{array}{ll}
g(Q(X,E_i,X),E_{i})&=\frac12\big( (\{z,y_i,z\}-\{t,x_i,z\},y_i)-(\{z,y_i,t\}-\{t,x_i,t\},x_i)  \big)\\
&=\frac12\Big(-(z,y_i)^2-(x_i,t)^2-2(t,z)(x_i,y_i)+2(x_i,z)(y_i,t)- \\
&\qquad -([z,y_i,z],y_i)+([t,x_i,z],y_i)+([z,y_i,t],x_i)-([t,x_i,t],x_i)\Big),\\
\end{array}
$$
and, analogously,
$$
\begin{array}{ll}
g(Q(X,E_{i+m},X),E_{i+m})&=\frac12\Big(-(z,x_i)^2-(y_i,t)^2-2(t,z)(x_i,y_i)-2(y_i,z)(x_i,t)- \\
&\quad -([z,x_i,z],x_i)-([t,y_i,z],x_i)-([z,x_i,t],y_i)-([t,y_i,t],y_i)\Big),\\
g(Q(X,E_{i+2m},X),E_{i+2m})&=\frac12\Big(-(z,x_i)^2-(y_i,t)^2+2(t,z)(x_i,y_i)+2(y_i,z)(x_i,t)- \\
&\quad -([z,x_i,z],x_i)+([t,y_i,z],x_i)+([z,x_i,t],y_i)-([t,y_i,t],y_i)\Big),\\
g(Q(X,E_{i+3m},X),E_{i+3m})&=\frac12\Big(-(z,y_i)^2-(x_i,t)^2+2(t,z)(x_i,y_i)-2(x_i,z)(y_i,t)- \\
&\quad -([z,y_i,z],y_i)-([t,x_i,z],y_i)-([z,y_i,t],x_i)-([t,x_i,t],x_i)\Big).\\
\end{array}
$$
If we sum, for $i=1,\dots,m$, we get that the trace   $\tr Q(X,-,X)=S_0+S_1-S_2-S_3$ can be obtained as
\begin{equation}\label{eq_traza}
\begin{array}{rl}
\tr Q(X,-,X)=\sum_{i=1}^{m}&\big( ([t,x_i,z],y_i)-([t,y_i,z],x_i) +([z,y_i,t],x_i)-([z,x_i,t],y_i)\\
&  -4(t,z)(x_i,y_i)+2(x_i,z)(y_i,t)-2(y_i,z)(x_i,t)\big) .
\end{array}
\end{equation}
Now we use the properties defining a symplectic triple system in Def.~\ref{def1} to get     
\begin{equation}\label{eq_paratraza}
\begin{array}{ll}
([t,x_i,z],y_i)-([t,y_i,z],x_i)&\stackrel{\eqref{eq_cuatro}}=([t,x_i,y_i],z)-([t,y_i,x_i],z)\\
&=([t,x_i,y_i]-[t,y_i,x_i],z) \\
&\stackrel{\eqref{eq_dos}}=
 ((t,y_i)x_i-(t,x_i)y_i+2(x_i,y_i)t,z )\\
&=-(y_i,t)(x_i,z)+(x_i,t)(y_i,z)+2(x_i,y_i)(t,z).
\end{array}
\end{equation}
We replace in the above equation $x_i$ with $y_i$ and $z$ with $t$, to get
\begin{equation}\label{eq_paratraza2}
 ([z,y_i,t],x_i)-([z,x_i,t],y_i)=-(x_i,z)(y_i,t)+(y_i,z)(x_i,t)+2(y_i,x_i)(z,t).
\end{equation}
Finally, we substitute Eqs.~\eqref{eq_paratraza} and \eqref{eq_paratraza2} in Eq.~\eqref{eq_traza}, which directly gives
$
\tr Q(X,-,X)=\sum_{i=1}^{m}0=0,
$
which ends the proof.
\end{proof} 

 Now we get, as  a corollary, our main result.
\begin{corollary}\label{main}
Take $(M,g)$ the semi-Riemannian space given by 
\begin{itemize}
\item $M=G/H$ is a homogeneous space such that there is a real simple symplectic triple system $T$ with $(\Lie(G),\Lie(H))= (\g(T),\inder(T))$;
\item $g$ is the metric defined in Eq.~\eqref{eq_nuestrag}.
\end{itemize} Then $M$
 is an Einstein   manifold, since, for any $X,Y\in\mm$,
$$
\Ric(X,Y)=(\dim M-1)\,g(X,Y).
$$
Consequently, its scalar curvature is always positive, given by
$$
s^g=(\dim M-1)\dim M.
$$
\end{corollary}

\begin{proof} 
For any $X,Y\in\mm$, the Ricci tensor is 
$$\begin{array}{ll}\Ric(X,Y)=\tr R(-,X)Y&=g(X,Y)\tr\id_\mm-\tr g(-,Y)X+\tr Q(-,X,Y)\\
&=g(X,Y)\dim M-g(X,Y)-\tr Q(X,-,Y)\end{array}
.$$
In particular $\mm\times\mm\to\RR,\ (X,Y)\mapsto \tr Q(X,-,Y)$ is a bilinear symmetric map, since it is obtained as a difference of bilinear symmetric maps. This permits to prove that   it is a zero map only by checking that  every $X\in\mm$ is isotropic, that is, $\tr Q(X,-,X)=0$. Indeed, if $X\in V\otimes T$, this is true by Proposition~\ref{pr_lacuenta} (b); and, if $X\in\spf(V)$, it is clear that the endomorphism $Q(X,-,X)\in\gl(\mm)$ is identically zero by Lemma~\ref{le_defQ}.

The conclusion about the scalar curvature is immediate.
\end{proof}

\begin{remark}\label{re_freu}
We would like to point out that $(T,\{ \,,\,,\,\})$, with the   brace-product defined in Eq.~\eqref{eq_llave}, is a \emph{Freudenthal triple system} \cite{Meyberg}. It is well-known the close relationship between these two ternary structures, symplectic  triple systems and Freudenthal triple systems, but  don't forget that our {\it{very geometric} }tensor $Q$ is defined by
\[
Q(\xi+\sum_i a_i\otimes x_i, \xi'+\sum_jb_j\otimes y_j,\xi''+\sum_kc_k\otimes z_k)=\sum_{i,j,k}\langle a_i,b_j\rangle c_k\otimes \{x_i,y_j,z_k\}.
\]
\end{remark}\smallskip

Before I conclude, I should like to
add some additional final thoughts, because this paper poses more questions than those  it answers. We have studied some families of homogeneous manifolds which are Einstein. What manifolds are we speaking of, exactly? What else can be said about their topology and geometry? And taking into consideration that there is a 3-Sasakian geometry, not only 3-Sasakian homogeneous manifolds, a last question arises: is there a geometry whose homogeneous examples are precisely these manifolds based on symplectic triple systems?


\end{document}